\documentclass{amsart}
\usepackage[utf8x]{inputenc}

\usepackage{amsthm,amsmath,amssymb,amstext,amsfonts,color}
\usepackage{fullpage}
\usepackage{todonotes}
\usepackage{graphicx}
\allowdisplaybreaks

\newcommand{\field}[1]{\mathbb{#1}}

\newcommand{\N}{\field{N}}
\newcommand{\Z}{\field{Z}}

\numberwithin{equation}{section}
\newtheorem{theorem}{Theorem}[section]
\newtheorem{lemma}[theorem]{Lemma}
\newtheorem{corollary}[theorem]{Corollary}
\newtheorem{conjecture}[theorem]{Conjecture}

\theoremstyle{remark}

\newtheorem*{definition}{Definition}

\renewenvironment{proof}[1][Proof]{\begin{trivlist}
\item[\hskip \labelsep {\bfseries #1:}]}{\qed\end{trivlist}}

\title{Cylindric partitions and some new $A_2$ Rogers--Ramanujan identities}

\author{Sylvie Corteel}
\address{Department of Mathematics, University of California, Berkeley, USA}
\email{corteel@berkeley.edu}

\author{Jehanne Dousse}
\address{Univ Lyon, CNRS, Université Claude Bernard Lyon 1, UMR5208, Institut Camille Jordan, F-69622 Villeurbanne, France}
\email{dousse@math.cnrs.fr}

\author{Ali Kemal Uncu}
\address{Johann Radon Institute for Computational and Applied Mathematics, Austrian Academy of Science, Altenbergerstraße 69, A-4040 Linz, Austria}
\email{akuncu@risc.jku.at}

\begin{document}

\begin{abstract}
We study the generating functions for cylindric partitions with profile $(c_1,c_2,c_3)$ for all $c_1,c_2,c_3$ such that $c_1+c_2+c_3=5$. This allows us to discover and prove seven new $A_2$ Rogers--Ramanujan identities modulo $8$ with quadruple sums, related with
work of Andrews, Schilling, and Warnaar.
\end{abstract}

\maketitle

\section{Introduction and statement of results}
A partition of a positive integer $n$ is a non-increasing sequence of natural numbers whose sum is $n$. For example, the partitions of $4$ are $4, 3+1, 2+2, 2+1+1,$ and $ 1+1+1+1.$ 

The Rogers--Ramanujan identities \cite{RogersRamanujan}, given by
\begin{align}
\sum_{n\geq 0}\frac{q^{n^2}}{(q;q)_n}&=\prod_{n\geq0}\frac{1}{(q;q^5)_{\infty} (q^4;q^5)_{\infty}}, \label{eq:rr1}\\
\sum_{n\geq 0}\frac{q^{n^2+n}}{(q;q)_n}&=\prod_{n\geq0}\frac{1}{(q^2;q^5)_{\infty} (q^3;q^5)_{\infty}}, \label{eq:rr2}
\end{align} 
are probably the most famous identities in the theory of $q$-series and partitions. 
Here we used the standard $q$-series notation, for $n \in \mathbb{N} \cup \{\infty\}$ and $j \in \N$,
\begin{align*}
(a;q)_n &:= \prod_{k=0}^{n-1} (1-aq^k),\\
(a_1, \dots, a_j ; q)_n &:= (a_1;q)_n \cdots (a_j;q)_n.
\end{align*}
We also use the convention that $1/(q;q)_n=0$ for negative $n$.

The Rogers--Ramanujan identities were interpreted in terms of partitions by MacMahon \cite{MacMahon} and Schur \cite{SchurRR} independently, and can be formulated as follows.
\begin{theorem}[Rogers--Ramanujan identities, partition version]
\label{th:RRcomb}
Let $i=0$ or $1$. For every natural number $n$, the number of partitions of $n$ such that the difference between two consecutive parts is at least $2$ and the part $1$ appears at most $i$ times is equal to the number of partitions of $n$ into parts congruent to $\pm (2-i) \mod 5.$
\end{theorem}

In addition to Combinatorics and Number Theory, the Rogers--Ramanujan are related to several other fields of mathematics such as representation theory of affine Lie algebras \cite{Lepowsky,Lepowsky2}, mathematical physics \cite{Baxter}, and algebraic geometry \cite{Mourtada}, to name only a few.
They were proved and generalised in many ways over the years, see for example \cite{AndrewsGordon,Andrews89,BressoudRR,Bressoud83,CorteelRSK,Garsiamilne,Gordon61,GOW16,Pascadi}.
The book of A. Sills give a good introduction to the subject \cite{sills2017invitation}.

Among these generalizations, an important one is due to Gordon \cite{Gordon61}, who embedded Theorem \ref{th:RRcomb} in an infinite family of combinatorial identities.
\begin{theorem}[Gordon's identities]
\label{th:Gordon}
Let $r$ and $i$ be integers such that $r\geq 2$ and $1\leq i \leq r.$ Let $S_{r,i}(n)$ be the number of partitions $\lambda=\lambda_1+\lambda_2+\cdots+\lambda_s$ of $n$ such that $\lambda_{j}-\lambda_{j+r-1} \geq 2$ for all $j$, and at most $i-1$ of the $\lambda_j$ are equal to $1$. Let $P_{r,i}(n)$ be the number of partitions of $n$ whose parts are not congruent to $0,\pm i \mod  2r+1$.

Then for every positive integer $n$, we have 
$$S_{r,i}(n)=P_{r,i}(n).$$
\end{theorem}
The Rogers--Ramanujan identities correspond to the cases $r=i=2$ and $r=2$, $i=1$.

Andrews later extended Gordon's identities in a $q$-series form~\cite{AndrewsGordon}, which can be stated as follows.
\begin{theorem}[Andrews--Gordon identities]
\label{th:AGseries}
Let $r \geq 2$ and $1 \leq i \leq r$ be two integers. We have
\begin{equation}\label{eq:AGri}
\sum_{n_1\geq\dots\geq n_{r-1}\geq0}\frac{q^{n_1^2+\dots+n_{r-1}^2+n_{i}+\dots+n_{r-1}}}{(q)_{n_1}} {n_1\brack n_1-n_2}_q \cdots {n_{r-2}\brack n_{r-2}-n_{r-1}}_q=\frac{(q^{2r+1},q^{i},q^{2r-i+1};q^{2r+1})_\infty}{(q)_\infty},
\end{equation}
where
$${m+n \brack m}_q := \left\lbrace \begin{array}{ll}\frac{(q;q)_{m+n}}{(q;q)_m(q;q)_{n}},&\text{for }m, n \geq 0,\\ 0,&\text{otherwise,}\end{array}\right.$$
is the classical $q$-binomial coefficient.
\end{theorem}
Note that~\eqref{eq:rr1} (resp.~\eqref{eq:rr2}) is the particular case of~\eqref{eq:AGri} where $r=i=2$ (resp. $r=2$ and  $i=1$).

\medskip
The Andrews-Gordon identities can be proved using the powerful machinery of Bailey pairs, first introduced by Bailey \cite{Bailey} and later developed by Andrews  \cite{AndrewsBailey}. All the identities provable by using classical Bailey pairs are related to certain characters of the Lie algebra $A_1$ \cite{AndrewsSchillingWarnaar}, and are therefore called $A_1$ Rogers--Ramanujan identities. In the beginning of the 1990s, Milne and Lilly \cite{MilneLilly1,MilneLilly2} extended the Bailey lemma to the type $A_{n-1}$ for general $n$. However, even though their work led to many $q$-series identities, they did not discover any $A_{n-1}$ Rogers--Ramanujan identities.

\medskip
In 1999, Andrews, Schilling and Warnaar \cite{AndrewsSchillingWarnaar} were able to describe an $A_2$ Bailey lemma and Bailey lattice. This allowed them to prove three $A_2$ Rogers--Ramanujan identities, one of which is the following:
\begin{theorem}[Andrews--Schilling--Warnaar]
\label{th:ASW}
$$\sum_{n_1,n_2\geq 0} \frac{q^{n_1^2+n_2^2-n_1n_2}}{(q;q)_{n_1}} {2n_1+1\brack n_2}_q=\frac{1}{(q,q,q^3,q^4,q^6,q^6;q^7)_\infty}.$$
\end{theorem}
Shortly after, Warnaar \cite{warnaar2006hall} showed how to prove these identities using Hall Littlewood identities.

In their paper, Andrews, Schilling and Warnaar found several very general families of sum-product identities (Theorems 5.1, 5.3, and 5.4 of \cite{AndrewsSchillingWarnaar}). They are quite intricate to state, so we will not do it in full generality here. However, let us note the important fact that, after multiplication by $(q;q)_{\infty}$, the product side of these identities can be interpreted in terms of characters of the $W_3$ algebra. However, in general, after multiplication by $(q;q)_{\infty}$, the sum-side does not have obviously positive coefficients anymore, even though we know the coefficients are positive because of the character.
So far, the three $A_2$ Rogers--Ramanujan identities of Andrews, Schilling, and Warnaar (Theorem \ref{th:ASW} and two other identities of the same shape) and a fourth identity of the same family of Corteel and Welsh \cite{CorteelWelsh} were the only cases in which the character could be expressed as a sum with obviously positive coefficients.

For example, the particular case $k=3$, $i=1$ in Theorem 5.3 of \cite{AndrewsSchillingWarnaar} can be stated as follows.
\begin{theorem}[Andrews--Schilling--Warnaar]
\label{th:ASW500}
$$(q,q)_{\infty} \sum_{a_1,b_1,a_2,b_2\in\Z} \frac{q^{a_1^2  + b_1^2 + a_2^2  + b_2^2 - a_1 b_1 + a_2 b_2 + a_1 + a_2 + b_1 + b_2}}{(q;q)_{a_1-a_2}(q;q)_{b_1-b_2}(q;q)_{a_2}(q;q)_{b_2}(q;q)_{a_2+b_2+1}} =\frac{1}{(q^2,q^3,q^3,q^4,q^4,q^5,q^5,q^6;q^8)_\infty}$$
\end{theorem}
The product on the right represents a character, but the $(q,q)_{\infty}$ prevents us from interpreting the left-hand side as a sum with positive coefficients.

\medskip

In 2016, Foda and Welsh \cite{FodaWelsh} had the brilliant idea to prove the Andrews--Gordon identities and the Bressoud identities using the combinatorics of cylindric partitions. We will define these objects in Section 2.
Soon after, Corteel \cite{CorteelRSK} gave a combinatorial proof of the Rogers--Ramanujan identities using cylindric partitions and the Robinson--Schensted--Knuth correspondence.
Corteel and Welsh in \cite{CorteelWelsh} built on these ideas to give a short proof of Andrews, Schilling and Warnaar's three $A_2$ Rogers--Ramanujan identities and their fourth one, using the combinatorics of cylindric partitions as well.

\medskip
Our goal in this paper is to continue these lines of work by presenting a new family of $A_2$ Rogers--Ramanujan identities using cylindric partitions. These new identities also give an expression of Andrews, Schilling and Warnaar's products related to characters as sums with obviously positive coefficients.

We prove simultaneously seven new identities. The four simpler ones are stated below, while the three more complicated ones can be found in Theorem \ref{th:main}.
\begin{theorem}
We have
\label{th:mainintro}
\begin{align*}
&\sum_{n_1,n_2,n_3,n_4\geq 0 } \frac{q^{n_1^2+n_2^2+n_3^2+n_4^2+n_1+n_2+n_3+n_4-n_1n_2 + n_2n_4}}{(q;q)_{n_1}} {n_1\brack n_2}_q{n_1\brack n_4}_q{n_2 \brack n_3}_q=\frac{1}{(q^2,q^3,q^3,q^4,q^4,q^5,q^5,q^6;q^8)_\infty}, \nonumber \\
&\sum_{n_1,n_2,n_3,n_4\geq 0 } \frac{q^{n_1^2+n_2^2+n_3^2+n_4^2+n_2+n_3+n_4-n_1n_2 + n_2n_4}}{(q;q)_{n_1}} {n_1\brack n_2}_q{n_1\brack n_4}_q{n_2 \brack n_3}_q
= \frac{1}{(q,q^2,q^3,q^4,q^4,q^5,q^6,q^7;q^8)_\infty} , \nonumber \\
& \sum_{n_1,n_2,n_3,n_4\geq 0 } \frac{q^{n_1^2+n_2^2+n_3^2+n_4^2+n_3-n_1n_2 + n_2n_4}}{(q;q)_{n_1}} {n_1\brack n_2}_q{n_1\brack n_4}_q{n_2 \brack n_3}_q
= \frac{1}{(q,q,q^3,q^3,q^5,q^5,q^7,q^7;q^8)_\infty}, \nonumber \\
 &\sum_{n_1,n_2,n_3,n_4\geq 0 } \frac{q^{n_1^2+n_2^2+n_3^2+n_4^2-n_1n_2 + n_2n_4}}{(q;q)_{n_1}} {n_1\brack n_2}_q{n_1\brack n_4}_q{n_2 \brack n_3}_q 
= \frac{1}{(q,q,q^2,q^4,q^4,q^6,q^7,q^7;q^8)_\infty}. \nonumber\\
\end{align*}
\end{theorem}

The first equation of our theorem gives an expression of the product in Theorem \ref{th:ASW500} as a sum with obviously positive coefficients, and we give a combinatorial interpretation in terms of cylindric partitions. More generally, our full main result, Theorem \ref{th:main}, gives a positive sum expression for all the products in the case $k=3$ in Theorem 5.3 of Andrews, Schilling, and Warnaar's paper \cite{AndrewsSchillingWarnaar}.

The two challenges in proving Theorem \ref{th:main} were the following: first, we ``guessed" the sum-side by guessing the generating
functions for certain cylindric partitions with bounded entries. Second, we proved these identities simultaneously by using a result of Corteel and Welsh (see Theorem \ref{th:CW}). As it is quite intricate to handle quadruple sums,
we used a recent computer algebra implementation of Ablinger and Uncu \cite{qFunctions} and we give an automated proof of the main result.

\medskip

This paper is organized as follows. In Section \ref{sec:next} we introduce cylindric partitions and the methodology.
To give a simple example on how to use cylindric partitions to prove partitions identities, we prove the classical Rogers--Ramanujan identities. In Section \ref{sec:proof}, we prove Theorem \ref{th:main}.

\subsection*{Acknowledgements} JD wants to thank the hospitality of the department of Mathematics 
at UC Berkeley where some of the results were proven. JD is partially funded by the ANR COMBIN\'e
ANR-19-CE48-0011. The research of AU is supported by the Austrian Science Fund (FWF) SFB50-11 Project.
The authors thank Ole Warnaar for his interest and remarks.

\section{Introduction and methodology}
\label{sec:next}

Cylindric partitions were introduced by Gessel and Krattenthaler \cite{GesselKrattenthaler}. They are defined as follows.

\begin{definition}\label{def:cylin} Let $k$ and $\ell$ be positive integers. Let $c=(c_1,c_2,\dots, c_k)$ be a composition, where $c_1+c_2+\dots+c_k=\ell$. A \emph{cylindric partition with profile $c$} is a vector partition $\Lambda = (\lambda^{(1)},\lambda^{(2)},\dots,\lambda^{(k)})$, where each $\lambda^{(i)} = \lambda^{(i)}_1+\lambda^{(i)}_2 + \cdots +\lambda^{(i)}_{s_i}$ is a partition, such that for all $i$ and $j$,
$$\lambda^{(i)}_j\geq \lambda^{(i+1)}_{j+c_{i+1}} \quad \text{and} \quad \lambda^{(k)}_{j}\geq\lambda^{(1)}_{j+c_1}.$$
\end{definition}

We define the size $|\Lambda|$ of a cylindric partition $\Lambda = (\lambda^{(1)},\lambda^{(2)},\dots,\lambda^{(k)})$ to be the sum of all the parts in the partitions $\lambda^{(1)},\lambda^{(2)},\dots,\lambda^{(k)}$.
We also define the largest part of a cylindric partition $\Lambda$ to be the maximum part among all the partitions in $\Lambda$, and denote it by $\max(\Lambda)$. 
We now define the bivariate generating function for cylindric partitions 
$$F_c(z,q):=\sum_{\Lambda\in \mathcal{P}_c} z^{\max{(\Lambda)}}q^{|\Lambda |},$$
where $\mathcal{P}_c$ denotes the set of all cylindric partitions with profile $c$. 

In 2007, Borodin \cite{Borodin} showed that when one sets $z=1$ in this bivariate generating function, it becomes a beautiful infinite product. 

\begin{theorem}[Borodin, 2007]
\label{th:Borodin}
Let $k$ and $\ell$ be positive integers, and let $c=(c_1,c_2,\dots,c_k)$ be a composition of $\ell$. Define $t:=k+\ell$ and $s(i,j) := c_i+c_{i+1}+\dots+ c_j$. Then,
\begin{equation}
\label{BorodinProd}
F_c(1,q) = \frac{1}{(q^t;q^t)_\infty} \prod_{i=1}^k \prod_{j=i}^k \prod_{m=1}^{c_i} \frac{1}{(q^{m+j-i+s(i+1,j)};q^t)_\infty} \prod_{i=2}^k \prod_{j=2}^i \prod_{m=1}^{c_i} \frac{1}{(q^{t-m+j-i-s(j,i-1)};q^t)_\infty}.
\end{equation}
\end{theorem}

In other words, if we denote by $\mu$ the partition $s(1,1)+s(1,2)+\cdots + s(1,k)$, and by $\mu^c$ its complement inside the $k \times s(1,k)$ rectangle, we have 
$$F_c(1,q) = \frac{1}{(q^t;q^t)_\infty} \prod_{\square \in \mu} \frac{1}{(q^{h(\square)};q^t)_\infty} \prod_{\square \in \mu^c} \frac{1}{(q^{t-h(\square)};q^t)_\infty},$$
where $h(\square)$ denotes the hook length of the box $\square$, see Figure \ref{fig:borodin}.

\begin{figure}[h]
\includegraphics[width=0.5\textwidth]{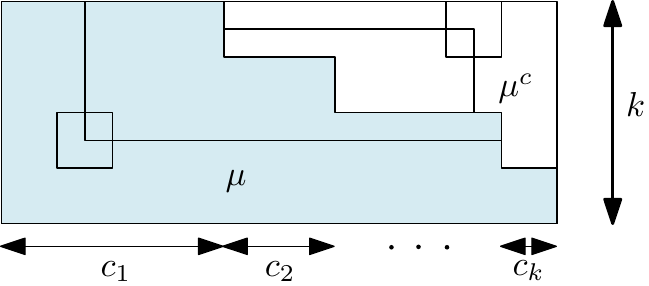}
\caption{The Borodin formula in terms of hooks}
\label{fig:borodin}
\end{figure}

Theorem \ref{th:Borodin} is very useful for computing generating functions for cylindric partitions. For example, using the hooks given in Figure \ref{fig:borodin311}, we can easily obtain that the generating function for cylindric partitions with profile $(3,1,1)$ is
$$F_{(3,1,1)}(1,q) = \frac{1}{(q;q)_{\infty}} \times \frac{1}{(q,q,q^3,q^3,q^5,q^5,q^7,q^7;q^8)_\infty}.$$

\begin{figure}[h]
\includegraphics[width=0.2\textwidth]{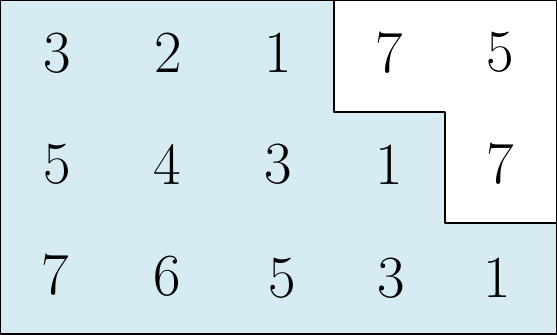}
\caption{The Borodin formula for profile $(3,1,1)$}
\label{fig:borodin311}
\end{figure}

Note that using the definition of cylindric partitions it is trivial to check that
$$F_c(z,q)=F_{S(c)}(z,q)$$ where $S(c)=(c_2,\ldots ,c_k,c_1)$.
In the case $k=3$, we have an extra identity.
\begin{corollary}
For any positive integers $c_1,c_2,c_3$, we have
$$
F_{(c_1,c_2,c_3)}(1,q)=F_{(c_2,c_1,c_3)}(1,q).
$$
\end{corollary}
\begin{proof}
We can assume without loss of generality that $c_1 > c_2.$ On Figure \ref{fig:proof}, we can see that the hooks in the plain colored areas are exactly the same for both profiles. 

\begin{figure}[h]
\includegraphics[width=1\textwidth]{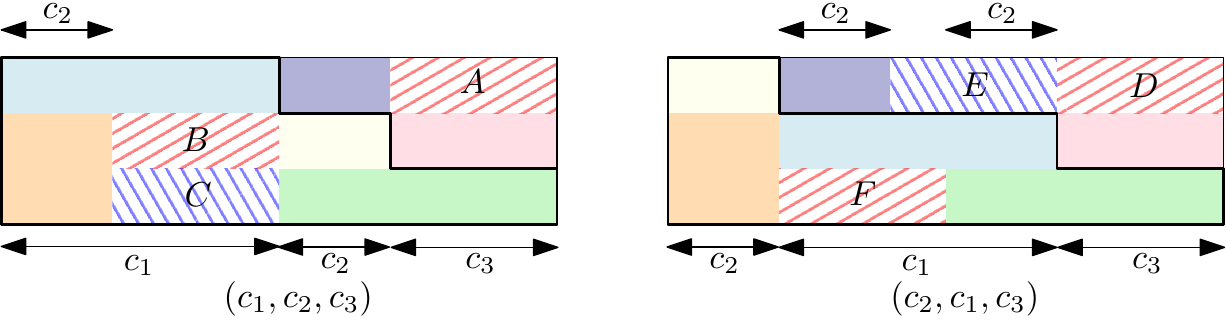}
\caption{The Borodin formula for profiles $(c_1,c_2,c_3)$ and $(c_2,c_1,c_3)$}
\label{fig:proof}
\end{figure}

It remains to see what happens in the dashed areas. First, we have
\begin{equation*}
\prod_{\square \in C} \frac{1}{(q^{h(\square)};q^t)_\infty} =\prod_{\square \in E} \frac{1}{(q^{t-h(\square)};q^t)_\infty}.
\end{equation*}
Indeed,
$\{ h(\square) | \square \in C\}= \{ c_2+c_3+3, \dots, c_1+c_3+2\}.$ On the other hand,
$\{ h(\square) | \square \in E\}= \{c_2+1, \dots, c_1\},$ so that
$\{ t-h(\square) | \square \in E\}= \{ c_2+c_3+3, \dots, c_1+c_3+2\}.$

Finally, the only thing left to do is proving that
\begin{equation}
\label{eq:AB=DF}
\prod_{\square \in B} \frac{1}{(q^{h(\square)};q^t)_\infty} \prod_{\square \in A} \frac{1}{(q^{t-h(\square)};q^t)_\infty} =\prod_{\square \in F} \frac{1}{(q^{h(\square)};q^t)_\infty} \prod_{\square \in D} \frac{1}{(q^{t-h(\square)};q^t)_\infty}.
\end{equation}
Indeed, we have:
\begin{align*}
\{ h(\square) | \square \in B\} &= \{ c_2+2, \dots, c_1+1\},\\
\{ t-h(\square) | \square \in A\} &= \{t- (c_2+2), \dots, t-( c_2+c_3+1)\}\\
&= \{c_1+2, \dots,  c_1+c_3+1 \}.
\end{align*}
On the side of the $(c_2,c_1,c_3)$ profile, we have:
\begin{align*}
\{ h(\square) | \square \in F\} &= \{ c_2+c_3+2, \dots, c_1+c_3+1\},\\
\{ t-h(\square) | \square \in D\} &= \{t- (c_1+2), \dots, t-( c_1+c_3+1)\}\\
&= \{c_2+2, \dots,  c_2+c_3+1 \}.
\end{align*}
Thus $\{ h(\square) | \square \in B\} \cup \{ t-h(\square) | \square \in A\} =\{ h(\square) | \square \in F\} \cup \{ t-h(\square) | \square \in D\}$ and  \eqref{eq:AB=DF} is proved.
\end{proof}

\medskip
On the other hand, Corteel and Welsh \cite{CorteelWelsh} showed that $F_c(z,q)$ satisfies a nice $q$-difference equation. Let $c=(c_1, \dots , c_k)$ and use the convention that $c_0=c_k$. Denote by $I_c$ the set of indices $j \in \{1, \dots, k\}$ such that $c_j >0$. Given a subset $J$ of $I_c$, the composition $c(J)=(c_1(J), \dots , c_k(J))$ is defined by:
\begin{equation}
\label{eq:c(J)}
c_i(J):= \begin{cases}
c_i-1 &\text{if $i \in J$ and $i-1 \notin J$},\\
c_i+1 &\text{if $i \notin J$ and $i-1 \in J$},\\
c_i &\text{otherwise}.
\end{cases}
\end{equation}
The $q$-difference equation is as follows.

\begin{theorem}[Corteel--Welsh] 
\label{th:CW}
For any profile $c$, 
\begin{equation}
\label{CorteelRec}
F_c(z,q) = \sum_{\emptyset\subset J\subseteq I_c} (-1)^{|J|-1} \frac{F_{c(J)}(z q^{|J|},q)}{(1-z q^{|J|})},
\end{equation}
with the initial conditions $F_c(0,q)=F_c(z,0)=1$.
\end{theorem}

Cylindric partitions, and in particular Theorems \ref{th:Borodin} and \ref{th:CW}, are a very interesting tool to discover and prove Rogers--Ramanujan type identities.

Before we prove our new results, we present a warm-up exercise that gives a quick proof of
the Rogers--Ramanujan identities, as in  \cite{CorteelRSK} .
These identities can be proved quite easily by using cylindric partitions with profiles $(3,0)$ and $(2,1)$. By Theorem \ref{th:Borodin}, we have
\begin{align}
\label{eq:prod30} F_{(3,0)}(1,q)&= \frac{1}{(q;q)_{\infty}} \times \frac{1}{(q^2,q^3;q^5)_{\infty}},\\
\label{eq:prod21} F_{(2,1)}(1,q)&= \frac{1}{(q;q)_{\infty}} \times \frac{1}{(q,q^4;q^5)_{\infty}}.
\end{align}
By Theorem \ref{th:CW}, we also have
\begin{align}
 \label{eq:30} F_{(3,0)}(z,q)&= \frac{F_{(2,1)}(zq,q)}{1-zq},\\
 \label{eq:21} F_{(2,1)}(z,q)&= \frac{F_{(3,0)}(zq,q)}{1-zq} +  \frac{F_{(2,1)}(zq,q)}{1-zq} - \frac{F_{(2,1)}(zq^2,q)}{1-zq^2}.
\end{align}

Writing, for every composition $c$,
$$G_c(z,q):= (zq;q)_{\infty} F_c(z,q),$$ 
and substituting \eqref{eq:30} into \eqref{eq:21}, we obtain
$$G_{(2,1)}(z,q)= G_{(2,1)}(zq,q)+ zq G_{(2,1)}(zq^2,q).$$
Thus we can easily show that
$$G_{(2,1)}(z,q)= \sum_{n \geq 0} \frac{z^nq^{n^2}}{(q;q)_n},$$
and deduce from \eqref{eq:30} that
$$G_{(3,0)}(z,q)= \sum_{n \geq 0} \frac{z^nq^{n^2+n}}{(q;q)_n}.$$
Combining this with \eqref{eq:prod30} and \eqref{eq:prod21} proves the Rogers--Ramanujan identities.

This method is crucial in the recent paper of Corteel and Welsh \cite{CorteelWelsh}, where they studied all cylindric partitions profiles where $k=3$ and $\ell=4$ to reprove the $A_2$ Rogers--Ramanujan identities due to Andrews, Schilling and Warnaar \cite{AndrewsSchillingWarnaar}.

\medskip
In this paper, we use Theorems \ref{th:Borodin} and \ref{th:CW} on cylindric partitions with all profiles $(c_1,c_2,c_3)$ such that
$c_1+c_2+c_3=5$, to prove new $A_2$ Rogers--Ramanujan type. Our results can be stated as follows.

\begin{theorem}
We have
\label{th:main}
\begin{align}
\label{eq:th500} G_{(5,0,0)}(1,q) &= \sum_{n_1,n_2,n_3,n_4\geq 0 } \frac{q^{n_1^2+n_2^2+n_3^2+n_4^2+n_1+n_2+n_3+n_4-n_1n_2 + n_2n_4}}{(q;q)_{n_1}} {n_1\brack n_2}_q{n_1\brack n_4}_q{n_2 \brack n_3}_q\\
&= \frac{1}{(q^2,q^3,q^3,q^4,q^4,q^5,q^5,q^6;q^8)_\infty}, \nonumber \\
\label{eq:th410=401} G_{(4,1,0)}(1,q) &=\sum_{n_1,n_2,n_3,n_4\geq 0 } \frac{q^{n_1^2+n_2^2+n_3^2+n_4^2+n_2+n_3+n_4-n_1n_2 + n_2n_4}}{(q;q)_{n_1}} {n_1\brack n_2}_q{n_1\brack n_4}_q{n_2 \brack n_3}_q
\\ \nonumber =G_{(4,0,1)}(1,q) &=\sum_{n_1,n_2,n_3,n_4\geq 0 } \frac{q^{n_1^2+n_2^2+n_3^2+n_4^2+n_1+n_3-n_1n_2 + n_2n_4} \left( 1 + q^{n_1+n_2+n_4+1} \right)}{(q;q)_{n_1}} {n_1\brack n_2}_q{n_1\brack n_4}_q{n_2 \brack n_3}_q
\\& = \frac{1}{(q,q^2,q^3,q^4,q^4,q^5,q^6,q^7;q^8)_\infty} , \nonumber \\
\label{eq:th320=302} G_{(3,0,2)}(1,q) &=\sum_{n_1,n_2,n_3,n_4\geq 0 } \frac{q^{n_1^2+n_2^2+n_3^2+n_4^2+n_1-n_1n_2 + n_2n_4} \left( 1+q^{n_1+n_3+1}+q^{2n_1+n_2+n_3+n_4+2} \right)}{(q;q)_{n_1}} \\
\nonumber &\qquad \qquad \qquad \times {n_1\brack n_2}_q{n_1\brack n_4}_q{n_2 \brack n_3}_q \\
=G_{(3,2,0)}(1,q) &=\sum_{n_1,n_2,n_3,n_4\geq 0 } \frac{q^{n_1^2+n_2^2+n_3^2+n_4^2+n_1-n_1n_2 + n_2n_4} \left( q^{n_3}+q^{n_1+1}+q^{2n_1+n_3+2}+q^{3n_1+n_2+n_3+n_4+3} \right)}{(q;q)_{n_1}} \\
\nonumber &\qquad \qquad \qquad \times {n_1\brack n_2}_q{n_1\brack n_4}_q{n_2 \brack n_3}_q \nonumber\\
&= \frac{1}{(q,q^2,q^2,q^3,q^5,q^6,q^6,q^7;q^8)_\infty}, \nonumber \\
\label{eq:th311} G_{(3,1,1)}(1,q) &=\sum_{n_1,n_2,n_3,n_4\geq 0 } \frac{q^{n_1^2+n_2^2+n_3^2+n_4^2+n_3-n_1n_2 + n_2n_4}}{(q;q)_{n_1}} {n_1\brack n_2}_q{n_1\brack n_4}_q{n_2 \brack n_3}_q
\\& = \frac{1}{(q,q,q^3,q^3,q^5,q^5,q^7,q^7;q^8)_\infty}, \nonumber \\
\label{eq:th221} G_{(2,2,1)}(1,q) &=\sum_{n_1,n_2,n_3,n_4\geq 0 } \frac{q^{n_1^2+n_2^2+n_3^2+n_4^2-n_1n_2 + n_2n_4}}{(q;q)_{n_1}} {n_1\brack n_2}_q{n_1\brack n_4}_q{n_2 \brack n_3}_q 
\\&= \frac{1}{(q,q,q^2,q^4,q^4,q^6,q^7,q^7;q^8)_\infty}. \nonumber
\end{align}
\end{theorem}

%
%
%
%

\section{Proof of Theorem \ref{th:main}}
\label{sec:proof}
In this section, we prove our Rogers--Ramanujan type identities. For the product side, a direct application of Theorem \ref{th:Borodin} gives the following.
\begin{corollary}
We have
\label{th:product_side}
\begin{align}
\label{eq:prod500} G_{(5,0,0)}(1,q) &= \frac{1}{(q^2,q^3,q^3,q^4,q^4,q^5,q^5,q^6;q^8)_\infty},\\
\label{eq:prod410=401} G_{(4,1,0)}(1,q) =G_{(4,0,1)}(1,q) &= \frac{1}{(q,q^2,q^3,q^4,q^4,q^5,q^6,q^7;q^8)_\infty} , \\
\label{eq:prod320=302} G_{(3,0,2)}(1,q) =G_{(3,2,0)}(1,q) &= \frac{1}{(q,q^2,q^2,q^3,q^5,q^6,q^6,q^7;q^8)_\infty},\\
\label{eq:prod311} G_{(3,1,1)}(1,q) &=\frac{1}{(q,q,q^3,q^3,q^5,q^5,q^7,q^7;q^8)_\infty},  \\
\label{eq:prod221} G_{(2,2,1)}(1,q) &= \frac{1}{(q,q,q^2,q^4,q^4,q^6,q^7,q^7;q^8)_\infty}.
\end{align}
\end{corollary}

We will now use Theorem \ref{th:CW} to prove the sum sides. All computations are more easily done with $G_c(z,q)$ than $F_c(z,q)$, so let us first recall a reformulation of Theorem \ref{th:CW} given in \cite{CorteelWelsh}.

\begin{theorem}[Corteel and Welsh, 2019] 
\label{th:CWforG}
For any profile $c$, 
\begin{equation}
\label{eq:qdiffG}
G_c(z,q) = \sum_{\emptyset\subset J\subseteq I_c} (-1)^{|J|-1} (zq;q)_{|J|-1} G_{c(J)}(z q^{|J|},q),
\end{equation}
with the initial conditions $G_c(0,q)=G_c(z,0)=1$.
\end{theorem}

Our main result is the following expressions for $G_c(z,q)$ as quadruple series.
\begin{theorem}
We have
\label{th:sum_side}
\begin{align}
\label{eq:sum500} G_{(5,0,0)}(z,q) &= \sum_{n_1,n_2,n_3,n_4\geq 0 } \frac{z^{n_1}q^{n_1^2+n_2^2+n_3^2+n_4^2+n_1+n_2+n_3+n_4-n_1n_2 + n_2n_4}}{(q;q)_{n_1}} {n_1\brack n_2}_q{n_1\brack n_4}_q{n_2 \brack n_3}_q,\\
\label{eq:sum410} G_{(4,1,0)}(z,q) &=\sum_{n_1,n_2,n_3,n_4\geq 0 } \frac{z^{n_1}q^{n_1^2+n_2^2+n_3^2+n_4^2+n_2+n_3+n_4-n_1n_2 + n_2n_4}}{(q;q)_{n_1}} {n_1\brack n_2}_q{n_1\brack n_4}_q{n_2 \brack n_3}_q,\\
\label{eq:sum401} G_{(4,0,1)}(z,q) &=\sum_{n_1,n_2,n_3,n_4\geq 0 } \frac{z^{n_1}q^{n_1^2+n_2^2+n_3^2+n_4^2+n_1+n_3-n_1n_2 + n_2n_4} \left( 1 + zq^{n_1+n_2+n_4+1} \right)}{(q;q)_{n_1}} {n_1\brack n_2}_q{n_1\brack n_4}_q{n_2 \brack n_3}_q\\
\label{eq:sum302} G_{(3,0,2)}(z,q) &=\sum_{n_1,n_2,n_3,n_4\geq 0 } \frac{z^{n_1} q^{n_1^2+n_2^2+n_3^2+n_4^2+n_1-n_1n_2 + n_2n_4} \left( 1+zq^{n_1+n_3+1}+zq^{2n_1+n_2+n_3+n_4+2} \right)}{(q;q)_{n_1}} \\
\nonumber &\qquad \qquad \qquad \times {n_1\brack n_2}_q{n_1\brack n_4}_q{n_2 \brack n_3}_q, \\
\label{eq:sum320} G_{(3,2,0)}(z,q) &=\sum_{n_1,n_2,n_3,n_4\geq 0 } \frac{z^{n_1}q^{n_1^2+n_2^2+n_3^2+n_4^2+n_1-n_1n_2 + n_2n_4} \left( q^{n_3}+zq^{n_1+1}+zq^{2n_1+n_3+2}+zq^{3n_1+n_2+n_3+n_4+3} \right)}{(q;q)_{n_1}} \\
\nonumber &\qquad \qquad \qquad \times {n_1\brack n_2}_q{n_1\brack n_4}_q{n_2 \brack n_3}_q,\\
\label{eq:sum311} G_{(3,1,1)}(z,q) &=\sum_{n_1,n_2,n_3,n_4\geq 0 } \frac{z^{n_1} q^{n_1^2+n_2^2+n_3^2+n_4^2+n_3-n_1n_2 + n_2n_4}}{(q;q)_{n_1}} {n_1\brack n_2}_q{n_1\brack n_4}_q{n_2 \brack n_3}_q,\\
\label{eq:sum221} G_{(2,2,1)}(z,q) &=\sum_{n_1,n_2,n_3,n_4\geq 0 } \frac{z^{n_1}q^{n_1^2+n_2^2+n_3^2+n_4^2-n_1n_2 + n_2n_4}}{(q;q)_{n_1}} {n_1\brack n_2}_q{n_1\brack n_4}_q{n_2 \brack n_3}_q.
\end{align}
\end{theorem}
Note that even though, when $z=1$,  $G_{(4,1,0)}(1,q)= G_{(4,0,1)}(1,q)$ and $G_{(3,0,2)}(1,q)=G_{(3,2,0)}(1,q)$, in general $G_{(4,1,0)}(z,q)$ is not equal to $G_{(4,0,1)}(z,q)$ and $G_{(3,0,2)}(z,q)$ is not equal to $G_{(3,2,0)}(z,q)$.

\begin{proof}
Applying Theorem \ref{th:CWforG}, we obtain the following $q$-difference equations for the compositions under consideration.
\begin{align}
\label{eq:qdiff500} G_{(5,0,0)}(z,q) &= G_{(4,1,0)}(zq,q),\\
\label{eq:qdiff410} G_{(4,1,0)}(z,q) &= G_{(4,0,1)}(zq,q) + G_{(3,2,0)}(zq,q) - (1-zq) G_{(3,1,1)}(zq^2,q),\\
\label{eq:qdiff401} G_{(4,0,1)}(z,q) &= G_{(5,0,0)}(zq,q) + G_{(3,1,1)}(zq,q) - (1-zq) G_{(4,1,0)}(zq^2,q),\\
\label{eq:qdiff320} G_{(3,2,0)}(z,q) &= G_{(3,1,1)}(zq,q) + G_{(3,0,2)}(zq,q) - (1-zq) G_{(2,2,1)}(zq^2,q),\\
\label{eq:qdiff311} G_{(3,1,1)}(z,q) &= G_{(4,1,0)}(zq,q) + G_{(3,0,2)}(zq,q) + G_{(2,2,1)}(zq,q)\\
\nonumber & - (1-zq) \big( G_{(4,0,1)}(zq^2,q) + G_{(3,2,0)}(zq^2,q) + G_{(2,2,1)}(zq^2,q) \big)\\
\nonumber & +(1-zq)(1-zq^2)  G_{(3,1,1)}(zq^3,q),\\
\label{eq:qdiff302} G_{(3,0,2)}(z,q) &= G_{(4,0,1)}(zq,q) + G_{(2,2,1)}(zq,q) - (1-zq) G_{(3,1,1)}(zq^2,q),\\
\label{eq:qdiff221} G_{(2,2,1)}(z,q) &= G_{(3,2,0)}(zq,q) + G_{(3,1,1)}(zq,q) + G_{(2,2,1)}(zq,q)\\
\nonumber & - (1-zq) \big( G_{(3,1,1)}(zq^2,q) + G_{(3,0,2)}(zq^2,q) + G_{(2,2,1)}(zq^2,q) \big)\\
\nonumber  & +(1-zq)(1-zq^2)  G_{(2,2,1)}(zq^3,q).
\end{align}

We use Ablinger and Uncu's \texttt{qFunctions} package \cite{qFunctions} together with Koutschan's \texttt{HolonomicFunctions} package \cite{HolonomicFunctions} to prove \eqref{eq:sum500}, \eqref{eq:sum410}, \eqref{eq:sum311} and \eqref{eq:sum221}. The coupled $q$-difference equation system \eqref{eq:qdiff500}-\eqref{eq:qdiff221} corresponds to a coupled $q$-recurrence system. To that end, for any profile $c$ we define the formal  \begin{equation}\label{eq:formalpowerseries_G_to_g} G_c(z,q) = \sum_{k\geq 0} g_c(k) z^k,
\end{equation} where $g_c(k)$ is a function of $q$. The initial conditions $g_c(0)=1$ and $g_c(k)=0$ for every $k\leq 0$ define these sequences uniquely for all the relevant profiles $c$. Then, for example \eqref{eq:qdiff410} is equivalent to \begin{equation}
\label{eq:qrec410} 
g_{(4,1,0)}(k) = q^k g_{(4,0,1)}(k) + q^k g_{(3,2,0)}(k) - q^{2k} g_{(3,1,1)}(k) + q^{2k-1}g_{(3,1,1)}(k-1),
\end{equation} for all $k\in\Z$. Uncoupling a coupled system of recurrences corresponds to Gaussian elimination and can be done automatically by symbolic computation implementations. Here we use the \texttt{HolonomicFunctions} package \cite{HolonomicFunctions}, which uses Groebner bases calculations to uncouple a given linear system of recurrences. Uncoupling the recurrences shows that the coefficients $g_{(5,0,0)}(k)$, $g_{(4,1,0)}(k)$, $g_{(3,2,0)}(k)$, $g_{(3,1,1)}(k)$, and $g_{(2,2,1)}(k)$ each satisfy recurrences of order 4 and $g_{(4,0,1)}(k)$ and $g_{(3,0,2)}(k)$ satisfy recurrences of order 6. We present the recurrence that $g_{(4,1,0)}(k)$ satisfies as an explicit example:
\begin{align}
\label{eq:recg410}
a_k(q) &g_{(4,1,0)}(k) - b_k(q) g_{(4,1,0)}(k-1) - c_k(q) g_{(4,1,0)}(k-2) + d_k(q)g_{(4,1,0)}(k-3) + e_k(q) g_{(4,1,0)}(k-4)=0,\\ \intertext{where}
\nonumber a_k(q) &= (1-q^k) (1 - q^{k-2} - 
 2 q^{k-1} - q^k + q^{2 k-3} + q^{2 k-2} + q^{2 k-1} - q^{3k-5} - q^{3 k-4} - q^{3 k-3}\\ \nonumber &\hspace{4cm} - q^{4 k-5} - q^{4 k-4}- q^{4 k-3} + q^{5 k-7} + q^{5 k-6} + q^{5 k-5} - q^{6 k-8} - q^{6 k-7} - q^{6 k-6}),\\
\nonumber b_k(q)&=q^{2k+1} (1+ q^{k-2}-q^{2 k-6}-2 q^{2 k-5}-4 q^{2 k-4}-5 q^{2 k-3}-5 q^{2 k-2}-2 q^{2k-1}-q^{2 k}+q^{3 k-7}+2 q^{3 k-6}
\\ \nonumber &\hspace{1.7cm}+4 q^{3 k-5}+5 q^{3 k-4}+6 q^{3 k-3}+5 q^{3 k-2}+3 q^{3k-1}+q^{3 k}-q^{4 k-8}-4 q^{4 k-7}-6 q^{4 k-6}-8 q^{4 k-5}
\\ \nonumber &\hspace{1.7cm}-7 q^{4 k-4}-7 q^{4 k-3}-4 q^{4 k-2}-2q^{4 k-1}+q^{5 k-8}+2 q^{5 k-7}+2 q^{5 k-6}+q^{5 k-5}-q^{5 k-2}+q^{6 k-9}
\\ \nonumber &\hspace{1.7cm}+q^{6 k-8}+2 q^{6k-7}+2 q^{6 k-6}+3 q^{6 k-5}+3 q^{6 k-4}+3 q^{6 k-3}+2 q^{6 k-2}-q^{7 k-10}-3 q^{7 k-9}
\\ \nonumber &\hspace{1.7cm}-5q^{7 k-8}-7 q^{7 k-7}-7 q^{7 k-6}-7 q^{7 k-5}-5 q^{7 k-4}-4 q^{7 k-3}-q^{7 k-2}+q^{8k-10}+3 q^{8 k-9}
\\ \nonumber &\hspace{1.7cm}+5 q^{8 k-8}+6 q^{8 k-7}+4 q^{8 k-6}+3 q^{8 k-5}+q^{8 k-4}+q^{8 k-3}-q^{9k-10}-2 q^{9 k-9}-3 q^{9 k-8}
\\ \nonumber &\hspace{1.7cm}-2 q^{9 k-7}-q^{9 k-6}-q^{10 k-9}-q^{10 k-8}-2 q^{10k-7}-q^{10 k-6}-q^{10 k-5}+q^{11 k-10}+2 q^{11 k-9}
\\ \nonumber &\hspace{8cm}+2 q^{11 k-8}+q^{11 k-7}-q^{12k-10}-q^{12 k-9}-q^{12 k-8}),\\
\nonumber c_k(q) &= q^{3k-2}(1+q-2 q^{k-1}-3 q^k-2 q^{k+1}-q^{k+2}+q^{2 k-2}+2 q^{2k-1}+2 q^{2 k}+q^{2 k+1}+q^{2 k+2}-q^{3 k-5}-q^{3 k-4}
\\ \nonumber &\hspace{1.3cm}-3 q^{3 k-3}-3 q^{3 k-2}-3 q^{3 k-1}-q^{3 k}-q^{3k+1}+q^{4 k-7}+2 q^{4 k-6}+3 q^{4 k-5}+3 q^{4 k-4}+3 q^{4 k-3}
\\ \nonumber &\hspace{1.3cm}+q^{4 k-2}-2 q^{4 k}-q^{4 k+1}-q^{4k+2}-q^{5 k-8}-2 q^{5 k-7}-3 q^{5 k-6}-2 q^{5 k-5}-3 q^{5 k-4}-q^{5 k-3}
\\ \nonumber &\hspace{1.3cm}+2 q^{5 k-1}+2 q^{5 k}+q^{5k+1}+q^{6 k-9}+3 q^{6 k-8}+4 q^{6 k-7}+4 q^{6 k-6}+2 q^{6 k-5}+q^{6 k-4}-2 q^{6 k-3}
\\ \nonumber &\hspace{1.3cm}-2q^{6 k-2}-4 q^{6 k-1}-2 q^{6 k}-q^{6 k+1}-q^{7 k-9}-q^{7 k-8}-q^{7 k-7}+q^{7 k-6}+2 q^{7 k-5}+3 q^{7k-4}+3 q^{7 k-3}
\\ \nonumber &\hspace{1.3cm}+2 q^{7 k-2}+q^{7 k-1}+q^{8 k-9}-2 q^{8 k-6}-2 q^{8 k-5}-3 q^{8 k-4}-2q^{8 k-3}-q^{8 k-2}+q^{9 k-9}+2 q^{9 k-8}
\\ \nonumber &\hspace{1.3cm}+3 q^{9 k-7}+3 q^{9 k-6}+2 q^{9 k-5}+q^{9k-4}-q^{10 k-9}-q^{10 k-8}-q^{10 k-7}+q^{11 k-9}+q^{11 k-8}+q^{11 k-7}),\\
\nonumber
d_k(q) &= q^{5k-4}(1-q^{k-3}-q^{k-2}-q^{k-1}+q^{2 k-7}+q^{2 k-3}-q^{2 k}-q^{3 k-8}-2 q^{3 k-7}-2 q^{3k-6}-2 q^{3 k-5}-2 q^{3 k-4}
\\ \nonumber &\hspace{1.6cm}-q^{3 k-3}+q^{3 k-1}+q^{3 k}+q^{4 k-8}+q^{4 k-7}+q^{4 k-6}-q^{4k-3}-q^{4 k-2}-q^{4 k-1}-q^{5 k-9}-q^{5 k-8}
\\ \nonumber &\hspace{1.6cm}-q^{5 k-7}-q^{5 k-6}-q^{5 k-5}-q^{5 k-4}-q^{6 k-8}-q^{6 k-7}-q^{6 k-6}-q^{6 k-5}-q^{6 k-4}-q^{6 k-3}+q^{7 k-9}
\\ \nonumber &\hspace{5.5cm}+q^{7 k-8}+q^{7 k-7}-q^{8k-9}-q^{8 k-8}-q^{8 k-7}-q^{8 k-6}-q^{8 k-5}-q^{8 k-4}),\\
\nonumber e_k(q) &= q^{6k-9} (1-q^{k-1}-2 q^k-q^{k+1}+q^{2 k-1}+q^{2 k}+q^{2 k+1}-q^{3 k-2}-q^{3k-1}-q^{3 k}-q^{4 k-1}-q^{4 k}-q^{4 k+1}
\\ \nonumber &\hspace{8.6cm}+q^{5 k-2}+q^{5 k-1}+q^{5 k}-q^{6 k-2}-q^{6 k-1}-q^{6 k}).
\end{align}

The recurrence \eqref{eq:recg410} alone shows that symbolic computation implementations such as \texttt{HolonomicFunctions} \cite{HolonomicFunctions} is a necessity for modern applications to perform tedious calculations with ease and eliminate any errors. Generating the coupled system of $q$-difference equations \eqref{eq:qdiff500}-\eqref{eq:qdiff221}, turning these $q$-difference equations into $q$-recurrences in the form of \eqref{eq:qrec410} and uncoupling of the recurrence system (implicitly using \texttt{HolonomicFunctions}) can be automatically done by the \texttt{qFunctions} \cite{qFunctions} package. 

Now we can turn our attention on the claimed explicit expression \eqref{eq:sum410} for $G_{(4,1,0)}(z,q)$. We can find a recurrence that the coefficients \[g_k(q):=\sum_{n_2,n_3,n_4\geq 0 } \frac{q^{k^2+n_2^2+n_3^2+n_4^2+n_2+n_3+n_4-kn_2 + n_2n_4}}{(q;q)_{k}} {k\brack n_2}_q{k\brack n_4}_q{n_2 \brack n_3}_q \] of $z^k$ in \eqref{eq:sum410} satisfy using \texttt{HolonomicFunctions} \cite{HolonomicFunctions}. This is done by a successive application of Zeilberger's creative telescoping algorithm. Here one sees that $g_k(q)$ satisfies the same recurrence \eqref{eq:recg410} that $g_{(4,1,0)}(k)$ satisfies. We can easily check that $g_k(q)$ and $g_{(4,1,0)}(k)$ satisfy the same initial conditions. This proves \eqref{eq:sum410}.

Same steps can be taken to prove \eqref{eq:sum500}, \eqref{eq:sum311}, and \eqref{eq:sum221}. In these cases, the uncoupled recurrences that come from \eqref{eq:qdiff500}-\eqref{eq:qdiff221} and the recurrences we get from the right-hand sides of \eqref{eq:sum500}, \eqref{eq:sum311}, and \eqref{eq:sum221} coincide. Checking the trivial initial conditions proves these claims. 

In the remaining cases, \eqref{eq:sum401}, \eqref{eq:sum302}, and \eqref{eq:sum320}, we cannot directly apply the available symbolic computation implementations. This is due to the extra polynomial factors that appear in these expressions. Technically speaking, we can still split these sums and find recurrences for each piece. Later we can use the closure properties of holonomic sequences under addition to find a recurrence that is satisfied by all pieces simultaneously. This gives us a recurrence of much higher order than necessary. These recurrences do not match the recurrences we find from the coupled system \eqref{eq:qdiff500}-\eqref{eq:qdiff221}. Therefore, to finish the proofs this way, we would need to use the closure properties once again. This is complicated and it can be avoided all together.

With the identities \eqref{eq:sum410}, \eqref{eq:sum311}, and \eqref{eq:sum221} proven, we now turn to the remaining sum formulas of Theorem \ref{th:sum_side}.
The formula \eqref{eq:sum500} follows easily from \eqref{eq:sum410} and \eqref{eq:qdiff500}.

Replacing $z$ by $zq$ in \eqref{eq:qdiff500}, we have
$$G_{(5,0,0)}(zq,q) = G_{(4,1,0)}(zq^2,q).$$
Substituting this into \eqref{eq:qdiff401} yields
\begin{equation}
\label{eq:qdiff401bis}
G_{(4,0,1)}(z,q) = G_{(3,1,1)}(zq,q) + zq G_{(4,1,0)}(zq^2,q).
\end{equation}
Using \eqref{eq:sum410} and \eqref{eq:sum311} in this new equation \eqref{eq:qdiff401bis}, we prove \eqref{eq:sum401}.

Now in \eqref{eq:qdiff302}, we replace $G_{(4,0,1)}(zq,q)$ by its expression given by \eqref{eq:qdiff401bis}, and obtain
\begin{equation}
\label{eq:qdiff302bis}
G_{(3,0,2)}(z,q) = G_{(2,2,1)}(zq,q) + zq G_{(3,1,1)}(zq^2,q) +zq^2 G_{(4,1,0)}(zq^3,q).
\end{equation}
So we can prove \eqref{eq:sum302} by using \eqref{eq:sum410}, \eqref{eq:sum311}, and \eqref{eq:sum221} in \eqref{eq:qdiff302bis}.

Finally, in \eqref{eq:qdiff320}, we replace $G_{(3,0,2)}(zq,q)$ by its expression given by \eqref{eq:qdiff302bis}. This gives
\begin{equation}
\label{eq:qdiff320bis}
G_{(3,2,0)}(z,q) = G_{(3,1,1)}(zq,q) + zq G_{(2,2,1)}(zq^2,q) +zq^2 G_{(3,1,1)}(zq^3,q)+ zq^3 G_{(4,1,0)}(zq^4,q).
\end{equation}
Thus we can prove the final sum \eqref{eq:sum320} by using \eqref{eq:sum410}, \eqref{eq:sum311}, and \eqref{eq:sum221} in \eqref{eq:qdiff320bis}.
\end{proof}

Combining Theorems \ref{th:product_side} and \ref{th:sum_side} concludes the proof of Theorem \ref{th:main}.


\section{Conclusion}

In this note we made a first progress towards the general family of $A_2$ Rogers--Ramanujan identities.
To get the full picture, we would need to compute the generating functions of cylindric partitions of
profile $(c_1,c_2,c_3)$ for any $c_1,c_2,c_3\in \N$. 
The Rogers--Ramanujan identities correspond to the case $c_1+c_2+c_3=2$. The 
previously known $A_2$ Rogers--Ramanujan identities of Andrews, Schilling, and Warnaar correspond to the case $c_1+c_2+c_3=4$. In this note
we settle the case $c_1+c_2+c_3=5$.
For now we do not have the right tools to conjecture the sum side for general  $(c_1,c_2,c_3)$ and we do not know
how to go further. Nevertheless the problem seems to have some beautiful structure.
Let us define
$
f_{n,c}(q)
$
to be the generating function of cylindric partitions with profile $c$ with entries bounded by $n$.
Or equivalently,
$$
f_{n,c}(q):=[z^n]\frac{F_c(z,q)}{1-z}.
$$
And let $g_{n,c}:=[z^n]G_c(z,q)$.
For example, using the combinatorics of lattice paths, one can prove the following.
\begin{lemma}
The generating function of cylindric partitions with profile $(\ell,0,0)$ with entries bounded by $n$
is
$$f_{n,(\ell,0,0)}(q)=\frac{P_{n,(\ell,0,0)}(q)}{(q^3;q^3)_n},$$
where $P$ is a polynomial in $q$ with integer coefficients
and $P_{n,(\ell,0,0)}(1)={\ell+2\choose 2}^n$.
\end{lemma}
See \cite{GesselKrattenthaler} to see how to write cylindric partitions as non-intersecting paths.

Using  Theorem \ref{th:CW}, we can prove that this holds in more generality.
The generating function for cylindric partitions with profile $c=(c_1,c_2,c_3)$ with entries bounded by $n$
is
$$f_{n,c}(q)=\frac{P_{n,c}(q)}{(q^3;q^3)_n},$$
where $P$ is a polynomial in $q$ with integer coefficients
and $P_{n,c}(1)={\ell+2\choose 2}^n$, where $\ell=c_1+c_2+c_3$.
We conjecture that
\begin{conjecture}
For any profile $c=(c_1,c_2,c_3)$, the polynomial $P_{n,c}(q)$
has non-negative coefficients. Moreover if $c_1+c_2+c_3\not\equiv 0 \mod 3$,
the generating function $g_{n,c}$ can be written as
$$
g_{n,c}=\frac{1}{(q;q)_{n}}Q_{n,c}(q),
$$
where $Q_{n,c}(q)$ is a polynomial in $q$ with non-negative coefficients
and $Q_{n,c}(1)=\left(\frac{(\ell+2)(\ell+1)}{6}-1\right)^n$, where $\ell=c_1+c_2+c_3$.
\end{conjecture}
This conjecture is true for any $n$ if $c_1+c_2+c_3=2$ (Roger-Ramanujan identities), $c_1+c_2+c_3=4$ ($A_2$ Rogers Ramanujan 
identities) and $c_1+c_2+c_3=5$ (this paper).
For example
$$
Q_{n,(2,2,1)}(q)=\sum_{n_2,n_3,n_4\geq 0 }q^{n^2+n_2^2+n_3^2+n_4^2-nn_2 + n_2n_4} {n\brack n_2}_q{n\brack n_4}_q{n_2 \brack n_3}_q
$$
and $Q_{n,(2,2,1)}(1)=6^n$.
This conjecture has been checked for small values of $k$ and $n$.

In our case $c_1+c_2+c_3=5$, we know that the product formula for $(c_1,c_2,c_3)$ is the same as the formula for any cyclic permutation
of the $c_i$'s. Nevertheless we did not find a $q$-series transformation that proves directly that the sum sides
are equal. The only way we prove this is by using the fact that both sums are equal to the same product.
For example, we leave as an open problem to give a direct proof of the fact that:
\begin{align*}
\sum_{n_1,n_2,n_3,n_4\geq 0 }& \frac{q^{n_1^2+n_2^2+n_3^2+n_4^2+n_2+n_3+n_4-n_1n_2 + n_2n_4}}{(q;q)_{n_1}} {n_1\brack n_2}_q{n_1\brack n_4}_q{n_2 \brack n_3}_q\\
 &=\sum_{n_1,n_2,n_3,n_4\geq 0 } \frac{q^{n_1^2+n_2^2+n_3^2+n_4^2+n_1+n_3-n_1n_2 + n_2n_4} \left( 1 + q^{n_1+n_2+n_4+1} \right)}{(q;q)_{n_1}} {n_1\brack n_2}_q{n_1\brack n_4}_q{n_2 \brack n_3}_q.
\end{align*}

Finally, we mentioned in the Introduction that Theorem 5.3 of Andrews, Schilling, and Warnaar \cite{AndrewsSchillingWarnaar} gives another expression of all the sums in our Theorem \ref{th:main}, but their sums have a factor $(q;q)_{\infty}$ in front of them, like in Theorem \ref{th:ASW500}. So far, the only way we know that their sums and ours are equal is also because they equal the same infinite products. It would be an interesting exercise to find a direct connection between the sums as well. For example, can one find a $q$-series proof of the equality
\begin{align*}
(q,q)_{\infty} &\sum_{a_1,b_1,a_2,b_2\in\Z} \frac{q^{a_1^2  + b_1^2 + a_2^2  + b_2^2 - a_1 b_1 + a_2 b_2 + a_1 + a_2 + b_1 + b_2}}{(q;q)_{a_1-a_2}(q;q)_{b_1-b_2}(q;q)_{a_2}(q;q)_{b_2}(q;q)_{a_2+b_2+1}} \\
&= \sum_{n_1,n_2,n_3,n_4\geq 0 } \frac{q^{n_1^2+n_2^2+n_3^2+n_4^2+n_1+n_2+n_3+n_4-n_1n_2 + n_2n_4}}{(q;q)_{n_1}} {n_1\brack n_2}_q{n_1\brack n_4}_q{n_2 \brack n_3}_q ?
\end{align*}

\bibliographystyle{alpha}
\bibliography{biblio}

\end{document}